\documentclass[11pt]{amsart}
\usepackage[letterpaper,hmargin=1.4in,vmargin=1.2in]{geometry}
\usepackage{amsmath, amssymb, amsthm, verbatim, amscd, graphicx}

\DeclareSymbolFont{AMSa}{U}{msa}{m}{n}
\DeclareMathSymbol{\square} {\mathord}{AMSa}{"03}

\DeclareSymbolFont{AMSb}{U}{msb}{m}{n}
\DeclareMathSymbol{\N}{\mathbin}{AMSb}{"4E}
\DeclareMathSymbol{\Z}{\mathbin}{AMSb}{"5A}
\DeclareMathSymbol{\R}{\mathbin}{AMSb}{"52}
\DeclareMathSymbol{\Q}{\mathbin}{AMSb}{"51}
\DeclareMathSymbol{\I}{\mathbin}{AMSb}{"49}
\DeclareMathSymbol{\C}{\mathbin}{AMSb}{"43}
\DeclareMathOperator{\sgn}{sgn}

\newtheorem{theorem}{Theorem}[section]
\newtheorem{lemma}[theorem]{Lemma}
\newtheorem{proposition}[theorem]{Proposition}
\newtheorem{corollary}[theorem]{Corollary}
\newtheorem{definition}[theorem]{Definition}
\newtheorem{remark}[theorem]{Remark}

\title{Fractional Loop Group and Twisted K-theory}
\author{Pedram Hekmati and Jouko Mickelsson}
\address{Department of Theoretical Physics, Royal Institute of Technology, SE-106
  91 Stockholm, Sweden; Department of Mathematics and Statistics,
  University of Helsinki, FI-00014 Helsinki, Finland}
\email{pedram@kth.se, jouko@kth.se}
\begin{document}

\begin{abstract}
We study the structure of abelian extensions of the group $L_qG$ of
$q$-differentiable loops (in the Sobolev sense), generalizing from the
case of central extension of the smooth loop group. This is motivated
by the aim of understanding the problems with current algebras in
higher dimensions. Highest weight modules are constructed for the Lie 
algebra. The construction is extended to the current algebra of 
supersymmetric Wess-Zumino-Witten model. An application to the twisted 
K-theory on $G$ is discussed.   
 
\end{abstract}

\maketitle

\section{Introduction}

The main motivation for the present paper comes from trying to
understand the representation theory of groups of gauge
transformations in higher dimensions than one. In the case of a
circle, the relevant group is the loop group $LG$ of smooth functions
on the unit circle $S^1$ taking values in a compact Lie group $G.$ In
quantum field theory one considers representations of a central
extension $\widehat{LG}$ of $LG;$ in case when $G$ is semisimple, this 
corresponds to an affine Lie algebra. The requirement that the energy is bounded
from below leads to the study of highest weight representations of
$\widehat{LG}.$ This part of the representation theory is well
understood, \cite{Kac}. 

In higher dimensions much less is known. Quantum field theory gives us 
a candidate for an extension of the gauge group $Map(M,G),$ the group
of smooth mappings from a compact manifold $M$ to a compact group $G.$ 
The extension is not central, but by an abelian ideal. The geometric
reason for this is that the curvature form of the determinant line bundle 
over the moduli space of gauge connections (the Chern class of which is 
determined by a quantum anomaly) is not homogeneous; it is not invariant 
under left (or right) translations,
\cite{Mic89}. 

There are two main obstructions when trying to extend the
representation theory of affine Lie algebras to the case of
$Map(M,\mathfrak g).$ The first is that there is no natural polarization giving 
meaning to the highest weight condition; on $S^1$ the polarization 
is given by the decomposition of loops to positive and negative 
Fourier modes. The second obstruction has to do with renormalization 
problems in higher dimensions. On the circle, with respect to the
Fourier polarization, one can use methods of canonical quantization 
for producing representations of the loop group; the only
renormalization needed is the normal ordering of quantities quadratic
in the fermion field, \cite{Lu76}, \cite{PrSe}. In higher dimensions 
further renormalization is needed, leading to an action of the gauge
group, not in a single Hilbert space, but in a Hilbert bundle over 
the space of gauge connections, \cite{Mic94}. 

In this paper we make partial progress in trying to resolve the two 
obstructions above. We consider instead of $LG$ the group $L_q G$ of
loops which are not smooth but only differentiable of order $0<q <
\infty$ in the Sobolev sense, the \it fractional loop group. \rm In the range 
$\frac12 \leq  q$ the usual theory of highest weight representations is 
valid, the cocycle determining the central extension is well defined down 
to the critical order $q=\frac12.$ However, for $q < \frac12$ we have 
to use again a ``renormalized'' cocycle defining an abelian extension 
of the group $L_q G,$ similar as in the case of $Map(M,G).$  
The renormalization means that the restriction of the 2-cocycle
to the smooth subgroup $LG \subset L_qG$ is equal to the 2-cocycle $c$ for
the central extension (affine Kac-Moody algebra) plus a coboundary $\delta \eta$
of a 1-cochain $\eta$ (the renormalization 1-cochain).
The 1-cochain is defined only on $LG$ and does not extend
to $L_qG$ when $q < 1/2,$ only  the sum $c+\delta \eta$ is well-defined
on $L_qG.$

The important
difference between $L_q G$ and $Map(M,G)$ is that in the former case
we still have  a natural polarization of the Lie algebra into positive 
and negative Fourier modes and we can still talk about highest weight 
modules for the Lie algebra $L_q \mathfrak g.$ 

Because of the existence of the highest weight modules for
$L_q\mathfrak g$ we can even define the supercharge operator $Q$ for 
the supersymmetric Wess-Zumino-Witten model. In the case of the
central extension the supercharge is
defined as a product of a fermion field on the circle and the gauge
current; this is well defined because the vacuum is annihilated both 
by the negative frequencies of the fermion field and the current, thus
when acting on the vacuum only the (finite number of) zero Fourier
modes remain. This property is still intact in the case of the
abelian extension of the current algebra. 

We can also introduce a family of supercharges, by a ``minimal coupling'' 
to a gauge connection on the loop group. In the case of central extension
the connections on $LG$ can be taken to be left invariant and they are
written as a fixed connection plus a left invariant 1-form $A$ on $LG$. 
The form $A$ at the identity element is identified as a vector in the 
dual $L\mathfrak g^*$ which again is identified, through an invariant
inner product, as a vector in $L\mathfrak g$. This vector in turn defines a 
$\mathfrak g$-valued 1-form on the circle. The left translations on $LG$ 
induce the gauge action on the potentials $A$. Modulo the action of
the group $\Omega G$ of based loops, the set of vector potentials on
the circle is equal to the group $G$ of holonomies. In this way 
the family of supercharges parametrized by $A$ defines an element 
in the twisted K-theory of $G.$ Here the twist is equal to an integral 
3-cohomology class on $G$ fixed by the level $k$ of the loop group 
representation, \cite{Mic02}. 

In the case of $L_q G$ and the abelian extension, the 
connections are not invariant under the action of $L_q G$ and thus we
have to consider the larger family of supercharges parametrized by the
space $\mathcal A$ of all connections of a circle bundle over $L_q G$. 
This is still an affine space, the extension of $L_q G$ acts on it. 
The family of supercharges transforms equivariantly under the extension 
and it follows that it can be viewed as an element in twisted K-theory 
of the moduli stack $\mathcal A//\widehat{L_q G}.$ This replaces the 
G-equivariant twisted K-theory on $\mathcal A//\widehat{LG}$ in case of the 
central extension $\widehat{LG},$ the latter being equivalent to twisted
G-equivariant K-theory on the group $G$ of gauge holonomies.      

The paper is organized as follows. In Section 2 we introduce the fractional
loop group and consider its role as the gauge group of a fractional Dirac-Yang-Mills
system. It turns out that the natural setting is a spectral triple in the sense of 
non-commutative geometry. Interestingly enough, similar attempts have been made 
recently, \cite{Her}. We move on to discuss the embedding $L_qG \subset GL_p$ 
in Section 3 and the construction of Lie algebra cocycles in Section 4. Finally the 
last two sections are devoted to extending the current algebra of the 
supersymmetric WZW model to the fractional case and discussing its application to the 
twisted K-theory of $G$. 

\section{Fractional Loop Group}

Let $G$ denote a compact semisimple Lie group and $\mathfrak{g}$ its Lie 
algebra. Fix a faithful representation $\rho:G\to GL(V)$ in a finite 
dimensional complex vector space $V$. 

\begin{definition} The fractional loop group $L_qG$ for real index $\frac 12 < q$ is defined  
to be the Sobolev space,
\[L_qG := H^q(S^1,G) = \{ g\in Map(S^1,G) \ | \ \left\|g\right\|_{2,q}^2 = 
\sum_{k\in \mathbb{Z}}(1+k^2)^q|\rho(g_k)|^2 <\infty\} \ ,\]
where $|\rho(g_k)|$ is the standard matrix norm of the $k$:th Fourier
component of $g:S^1 \to G$. The group operation is given by pointwise multiplication 
$(g_1g_2)(x) = g_1(x)g_2(x)$.
\end{definition}
There is a natural Hilbert Lie group structure on $L_qG$ for $\frac 12 < q$. It is defined 
by the Hilbert space completion of the Lie algebra of smooth maps $C^\infty(S^1,\mathfrak g)$ 
with respect to the Sobolev inner product. The exponential map $\text{exp}:H^q(S^1,\mathfrak g) \to H^q(S^1,G)$
provides a local chart near the identity and is extended to an atlas by left 
translations. For our purposes however, we will use a Banach topology on the Lie algebra,
\[L_q\mathfrak{g}= \{ X \in L^{\infty}(S^1,\mathfrak{g}) \ | \ ||X|| = ||X||_{\infty} +
||X||_{2,q} <\infty \ \} \ ,\]
where $L^{\infty}$ is the set of measurable essentially bounded functions and 
$|| \ ||_{\infty}$ denotes the supremum norm. This induces a Banach Lie group structure 
on $L_qG$, \cite{PrSe} and is a more appropriate topology for the applications in this paper.
In fact, we have natural inclusions of Lie groups $LG \subset L_qG\subset L^cG$, where $LG$ is 
the smooth loop group and $L^cG$ is the Banach-Lie group of continuous loops in $G$, see Section 6. 
In the next section, we will modify the definition 
of $L_qG$ with respect to a different norm, to allow for values $0 < q \leq \frac 12$. 

In Yang-Mills theory on the cylinder $S^1\times \mathbb{R}$, the loop group 
$LG$ appears as the group of (time-independent) local gauge transformations. 
Quantization of massless chiral fermions in external Yang-Mills fields breaks 
the local gauge symmetry, leading to a central extension of $LG$. In order 
to make sense of this in the fractional setting, we need a notion 
of fractional differentiation. The study of fractional calculus dates back to 
early 18th century and a comprehensive review can be found in \cite{Sam}. The 
transition to fractional calculus is by no means unique. There are several 
competing definitions, but many are known to coincide on overlapping domains. 
For functions on the real line the Riemann-Liouville fractional 
derivative is defined by
\[D^q_a \psi(x) = \frac{d^n}{dx^n}\left\{\frac{1}{\Gamma(n-q)}\int_a^x
\frac{\psi(y)}{(x-y)^{q-n+1}}dy\right\}\] 
for $n>q$ and $x>a$. On the circle however this definition proves
inconvenient since periodic functions are 
not mapped onto periodic ones. An operator that do preserve periodicity is 
the Weyl fractional derivative, 
\[\widetilde{D}^q \psi(x) = \sum_{k\in \mathbb{Z}} (ik)^q \psi_k e^{ikx} = 
\sum_{k\in \mathbb{Z}} e^{\frac{iq\pi}{2}\sgn(k)}|k|^q \psi_k e^{ikx}\]
for all $q\in\mathbb{R}$. In fact, by extending to the real line one shows that 
$\widetilde{D}^q$ and $D^q_a$ coincide for $-1<q$ and 
$a=-\infty$ on an appropriate domain. For our purposes, we need a 
self-adjoint operator on a dense domain in $\mathcal{H}=L^2(S^1,V)$. 
The fractional Dirac operator on the circle is therefore defined by
\[D^q \psi(x)  = \sum_{k \in \mathbb{Z}} \sgn(k)|k|^q \psi_k e^{ikx} \ ,\]
where the complex phase has been replaced by the sign function. This
has the consequence that $D^q \circ D^r \neq D^{q+r}$, but we have instead 
$D^q\circ D^r=|D^{q+r}|$. For odd integers $q>1$, the fractional Dirac operator 
is simply the $q$-th power of the rotation operator $-i\frac d{dx}$ on the circle. 
The domain of $D^q$ is the Sobolev space $H^q(S^1,V)$. 
It is by construction an unbounded, self-adjoint operator with discrete 
spectrum $\{\sgn(k)|k|^q\}_{k \in \mathbb{Z}}$ and a complete set of eigenstates 
in $\mathcal{H}$. However, the Leibniz rule is no longer satisfied
\[D^q(\psi \phi) = \sum_{k,m\in \mathbb{Z}}\sgn(k)|k|^q \psi_m \phi_{k-m}e^{ikx} \neq \]
\[\neq (D^q\psi)\phi + \psi (D^q\phi) = \sum_{k,m\in \mathbb{Z}}
\Big(\sgn(m)|m|^q+\sgn(k-m)|k-m|^q\Big) \psi_m \phi_{k-m}e^{ikx} \ ,\]
unless $q=1$. 

We introduce interactions by imposing local gauge invariance. The covariant 
derivative $D^q_A = D^q+A$ should transform equivariantly under gauge transformations
\[g^{-1}D^q_A g = D^q + g^{-1}[D^q,g] + g^{-1}Ag = D^q_{A^g} \ .\]
This motivates the following definition of fractional Yang-Mills connections on the circle; 
\[A = \alpha [D^q, \beta], \ \ \ \alpha , \beta \in H^q(S^1,\mathfrak{g}) \ .\]
The fractional loop group $L_qG$ acts on $A$ by
\[(g, A)\mapsto A^g = g^{-1}Ag + g^{-1}[D^q,g]\]
and the infinitesimal gauge action is given by
\[(X,A) \mapsto \mathcal{L}_X A = [A,X] + [D^q,X] \ .\]
For values $\frac 12<q\leq1$, there is a geometric interpretation of this data in the 
non-commutative geometry sense. What we have is precisely a spectral triple, namely a Dirac operator
$D^q$, a Hilbert space $\mathcal{H}$ and an associative $*$-algebra $L_q\Bbb C$, \cite{Connes}. 
Here $L_q\Bbb C = H^q(S^1,\Bbb C)$ is an algebra for $\frac 12 < q$ by the Sobolev multiplication theorem.

\begin{proposition} $[D^q, X]$ is a bounded operator for all $X \in L_q\Bbb C$ 
and $0<q\leq1$,
\[\left\|[D^q,X]\right\| = \sup_{\substack{\psi\in \mathcal H \\ \left\|\psi \right\|=1 }} \left\|[D^q,X]\psi\right\|<\infty \ . \]
\end{proposition}
\begin{proof} By expanding in Fourier series, 
\[[D^q,X]\psi = \sum_{k,m\in \mathbb{Z}}X_m \psi_{k-m}\Big(\sgn(k)|k|^q-\sgn(k-m)|k-m|^q\Big)e^{i(k+m)x}\] 
it follows that
\begin{eqnarray*} 
\left\|[D^q,X]\psi\right\|^2 &=& \sum_{m,k\in \mathbb{Z}} |X_m \psi_{k-m}|^2\Big(\sgn(k)|k|^q-\sgn(k-m)|k-m|^q\Big)^2 \\
&=& \sum_{m,n\in \mathbb{Z}} |X_m\psi_n|^2\Big(\sgn(n+m)|n+m|^q-\sgn(n)|n|^q\Big)^2 \ .
\end{eqnarray*}
Since the sequence $|\psi_n|^2$ belongs to $l^1$, the sum converges if
\[\sum_{m\in \mathbb{Z}} |X_m|^2\Big(\sgn(n+m)|n+m|^q-\sgn(n)|n|^q\Big)^2 < C\]
is uniformly bounded by some constant $C$. To establish this, we rewrite
\[\sum_{m\in \mathbb{Z}} |X_m|^2\Big(\sgn(n+m)|n+m|^q-\sgn(n)|n|^q\Big)^2 =\]\[= \sum_{m\in \mathbb{Z}} |X_m|^2 m^{2q}
\left(\sgn\left(1+\frac{m}{n}\right)\Big|1+\frac{n}{m}\Big|^q-\Big|\frac{n}{m}\Big|^q\right)^2 \ .\]
For $0<q\leq 1$, this sum is bounded for all $n\in\Z$ since
\[f(x) = \sgn\left(1+\frac{1}{x}\right)|1+x|^q-|x|^q\]
is a bounded function on the real line. 
\end{proof}
We have an immediate corollary:
\begin{corollary} For $\frac 12 \leq q\leq 1$ the space $L_q\mathbb{C}$
is the algebra of essentially bounded measurable loops $X$ such that
$[D^q,X]$ is a bounded operator. 
\end{corollary} 
The inverse statement follows from the observation that taking
$\psi(x)=1,$ the constant loop in the proof above, the norm $||[D^q,X]\psi||$ is equal to $||X||_{2,q}$. 
 
Similarly one verifies that $[|D^q|,X]$ is bounded for all $X \in L_q\mathbb{C}$. 
Recall that a spectral triple is $p^+$-summable if $|D^q|^{-p}$ belongs to the Dixmier 
ideal $\mathcal{L}^{1+}$. This means that for some real number $p\geq 1$, 
\[\lim_{N\to \infty} \frac{1}{\log(N)}\sum_{k=1}^N \lambda_k < \infty\] 
where $\lambda_k \geq \lambda_{k+1} \geq \lambda_{k+2}\ldots$ are eigenvalues of 
$|D^q|^{-p}$ listed in descending order. In our case
\[\lambda_k = \frac{1}{k^{qp}}\]
for $k=1,2,\ldots$. This gives 
\[\lim_{N \to \infty} \frac{1}{\log(N)}\sum_{k=1}^N \frac{1}{k^{qp}}\] 
which is finite if and only if $qp\geq 1$. Moreover, since 
\[\left\|{\rm ad}^n_{|D|^q}(X)\psi\right\|^2 = \sum_{m,k\in \mathbb{Z}} |X_m\psi_{k-m}|^2\Big(|k|^q-|k-m|^q\Big)^{2n}\]
diverges for $n>1$, we conclude that the spectral triple is not tame. 

It is also interesting to note that even though the algebra $L_q\mathbb{C}$ 
is commutative, the spectral dimension $p$ of the circle is strictly larger than one when $q<1$.

Fractional differentiability has been studied systematically in the more 
general context of $\theta$-summable spectral triples in \cite{Jaffe99}. However,
their definition of fractional differentiability, although similar, differs from ours which 
is geared to the special case of loop groups and $L_p$-summable spectral 
triples.

\section{Embedding of $L_qG$ in $GL_p$}

When dealing with representations of the loop group $LG$, one is lead to consider central 
extensions by the circle 
\[1\to S^1\to\widehat{LG}\to LG\to 1 \ .\]  
In Fourier basis, the generators $S^a_n$ of the Lie algebra $\widehat{L\mathfrak{g}}$ satisfy 
\[[S^a_n, S^b_m]=\lambda^{abc}S^c_{n+m}+kn\delta^{ab}\delta_{n,-m}\]
where $k$ is the central element (represented as multiplication by a
scalar in an irreducible  representation). Here the upper index refers
to a normalized basis (with respect to an invariant nondegenerate bilinear form)
of the Lie algebra $\bold{g}$ of the group $G.$ The
$\lambda^{abc}$'s are the structure constants in this basis. We have
adopted the Einstein summation convention meaning that an index appearing twice in a term is summed
over all its possible values. 
When trying to extend the central extension to the fractional setting, one runs immediately into a problem. 
For infinite linear combinations of the Fourier modes $S^a_n$  the central term blows up. 
A precise condition for the divergence can be formulated by regarding $L_qG$ as a group of operators 
in a Hilbert space. Let $\rho:G\to GL(V)$ denote a representation of $G$ as in the previous section. 
Elements in the fractional loop group act as multiplication operators in $\mathcal{H}=L^2(S^1,V)$
by pointwise multiplication,
\[(M_g \psi)(x) = \rho(g(x))\psi(x)\]
for all $g\in L_qG, \ \psi \in \mathcal{H}$. In fact, $M:L_qG \to GL(\mathcal{H}), \ g\mapsto M_g$ 
defines a continuous embedding into 
the general linear group. This statement is somewhat crude however. Below we show that $L_qG$ is 
actually contained in a subgroup $GL_p$ of $GL$. Recall that the sign operator 
$\epsilon = \frac{D^q}{|D^q|}$ defines an orthogonal decomposition $\mathcal{H}=\mathcal{H}_+\oplus \mathcal{H}_-$ 
into positive and negative Fourier modes. We use the convention that the zero mode of $D^q$ is on the positive 
side of the spectrum. Introduce the Schatten class
\[L_{2p}=\{A \in \mathcal{B}(\mathcal{\mathcal{H}}) \ |\ \left\| A \right\|_{2p}= 
\left[{\rm Tr}(A^\dagger A)^{p}\right]^{\frac{1}{2p}} < \infty\}\]
which is a two-sided ideal in the algebra of bounded operators $\mathcal{B}(\mathcal{H})$. In particular,
the first Schatten class $L_1$ is the space of trace class operators and $L_2$ is the space of 
Hilbert-Schmidt operators. We will use 
an equivalent norm which is better suited for computations 
\[\left\| A \right\|_{2p}= \left[\sum_{k\in \mathbb{Z}}\left\| A \phi_k \right\|^{2p}\right]^{\frac{1}{2p}}\]
where $\{\phi_k\}_{k\in \mathbb{Z}}$ is an orthonormal basis in $\mathcal{H}$.
The subgroup $GL_p\subset GL(\mathcal{H})$ is defined by
\[GL_p = \{A \in GL(\mathcal{H}) \ | \ [\epsilon, A] \in L_{2p}\} \ .\]
Writing elements in $GL(\mathcal{H})$ in block form with respect to the Hilbert space polarization
\[A = \begin{pmatrix} A_{++} & A_{+-} \\ A_{-+} & A_{--}\end{pmatrix} \ ,\]
the condition 
\[[\epsilon, A] = 2\begin{pmatrix} 0 & A_{+-} \\ -A_{-+} & 0\end{pmatrix}\in L_{2p}\]
simply means that the off-diagonal blocks are not ``too large". Given the topology defined by the norm
\[\left\| |A| \right\|_p= \left\| A_{++} \right\| + \left\| A_{+-} \right\|_{2p} +\left\| A_{-+} \right\|_{2p} +\left\| A_{--} \right\| \ ,\]
where 
\[\left\| a \right\| = \sup_{\left\|\psi\right\|=1} \left\| a\psi\right\|\] 
denotes the operator norm, $GL_p$ is a Banach Lie group with the Lie algebra
\[\mathfrak{gl}_p = \{X\in \mathcal{B}(\mathcal{H}) \ | \ [\epsilon, X] \in L_{2p}\} \ .\]
\begin{proposition} If $p\geq \frac{1}{2q}$, then $L_qG$ is contained in  $GL_p$. \end{proposition}
\begin{proof} In order to avoid cumbersome notation, we write $g(x)$ 
instead of $\rho(g(x))$. Expanding in Fourier series 
$g(x)=\sum_{k\in \mathbb{Z}}g_k e^{ikx}$, we have
\[M_{g}e^{ikx} = \sum_{m\in \mathbb{Z}} g_{m}e^{i(m+k)x} = \sum_{m\in \mathbb{Z}} g_{m-k}e^{imx}\]
so that $(M_g)_{mk} = g_{m-k}$. Moreover 
\[\epsilon e^{ikx} = \begin{cases} \frac{k}{|k|} e^{ikx} &  k \neq 0 \\ e^{ikx} & k=0\end{cases}\]
We check that $\left\|[\epsilon, M_g]\right\|_{2p}$ is finite;
\begin{eqnarray*} 
\left\|[\epsilon, M_g]\right\|_{2p}^{2p} &=& \sum_{k\in \mathbb{Z}} \left\|[\epsilon,M_g]e^{ikx}\right\|^{2p} \\ 
&=& \sum_{m, k\in \mathbb{Z}}\left|\frac{m}{|m|}-\frac{k}{|k|}\right|^{2p}|g_{k-m} e^{imx}|^{2p} \\
&=& \sum_{n,k\in \mathbb{Z}}\left|\frac{(n+k)}{|n+k|}-\frac{k}{|k|}\right|^{2p}|g_n|^{2p} \\
&=& 2^{2p}\sum_{n\in \mathbb{Z}}|n||g_n|^{2p} \\
&=& 2^{2p}\sum_{n\in \mathbb{Z}}||n|^{\frac{1}{2p}}g_n|^{2p} \\
&\leq& 2^{2p}\sum_{n\in \mathbb{Z}}(n^2)^{\frac{1}{2p}}|g_n|^{2} + C_g 
\end{eqnarray*}
where $C_g$ is the finite part of the sum containing all terms where $|n|^{\frac{1}{2p}}|g_n|\geq 1$. 
Thus the sum converges if the Sobolev norm
\[\left\|g\right\|_{2, \frac{1}{2p}}^2 = \sum_{n \in \mathbb{Z}}(1+n^2)^{\frac{1}{2p}}|g_n|^{2}\]
is finite, which is the case when 
\[p\geq \frac{1}{2q} \ .\]
\end{proof} 
The same arguments apply without modification to the Lie algebra $L_q\mathfrak{g}\hookrightarrow \mathfrak{gl}_p$.
We see here precisely how the degree of differentiability $q$ is related to the Schatten index $p$. Furthermore,
this allows us to extend the definition of $L_qG$ to all real values $0<q<\infty$:

\begin{definition} 
	The fractional loop group $L_qG$ is defined to be the group of continuous loops contained 
	in $GL_p$, where $p= \text{max}\left\{\frac 12, \frac{1}{2q}\right\}.$
	We shall use the induced Banach structure on the Lie algebra $L_q\mathfrak g$ coming from 
	the embedding, defined by the norm
	\[\left\| X \right\|_\infty + \left\| |X| \right\|_p  \ .\]
	This endows $L_qG$ with a Banach Lie group structure.
\end{definition}

\begin{remark} \rm By abuse of notation, we use the same label $L_qG$ as in Definition 2.1 for this slightly larger fractional loop group.
	Actually, it follows by Proposition 3.1 that the group in Definition 2.1 is continuously embedded in $L_qG$, for $q>\frac 12$.
	For the remainder of this paper, we shall adopt Definition 3.2 for the fractional loop group.
	Also, note that for $q\geq 1$, $L_qG$ consists of continuous loops contained in $GL_{\frac 12}$, that is operators whose 
	off-diagonal blocks are trace class. This is the critical value for the Schatten index $p$, since for lower values than $\frac 12$ the spaces 
	$L_{2p}$ are no longer ideals.
\end{remark}

The cocycle defining the central extension $\widehat{L\mathfrak{g}}$ can be written
\[c_0(X,Y)=\frac{1}{8}{\rm Tr}\Big(\epsilon[[\epsilon,X],[\epsilon,Y]]\Big)={\rm Tr}\Big(X_{+-}Y_{-+}-Y_{+-}X_{-+}\Big)\]
for $X, Y \in L\mathfrak{g}$. This is finite so long as the off-diagonal blocks are Hilbert-Schmidt, 
i.e. $[\epsilon, X]$ and  $[\epsilon, Y]$ belong to $L_2$. However for $p>1$, corresponding to 
differentiability $q$ less than $\frac{1}{2}$, the operators $X_{+-}Y_{-+}$ and $Y_{+-}X_{-+}$ are no 
longer trace-class, because $a,b\in L_{2p}$ implies that $ab\in L_p$ by H\"older inequality. This 
is the reason behind the divergence. In order to make sense of the cocycle for higher $p$, regularization is 
required. 

We introduce the Grassmannian $Gr_p$ \cite{Mic89}, which is a smooth Banach manifold parametrized by idempotent hermitian 
operators $F$ such that $F-\epsilon \in L_{2p}$. Points on $Gr_p$ may also be thought of as closed subspaces 
$W\subset \mathcal{\mathcal{H}}$ such that the orthogonal projection ${\rm pr}_+:W\to \mathcal{H}_+$ is Fredholm 
and ${\rm pr}_-:W\to \mathcal{H}_-$ is in
$L_{2p}$. Let $\eta(X;F)$ denote a 1-cochain parametrized by points on $Gr_p$. For a suitable choice of $\eta$, 
adding the coboundary to the original cocycle
\[c_p(X,Y;F) = c_0(X,Y) + (\delta\eta)(X,Y;F) \ ,\]
we obtain a well-defined cocycle on $\mathfrak{gl}_p$ for some fixed
$p$. The Lie algebra cohomology coboundary operator $\delta$ is
defined by Palais' formula in Section 4. 
Although each term on the right diverges separately, 
the sum will be finite. In the case $p=2$, we could take
\[\eta(X;F)=-\frac{1}{16}{\rm Tr}\Big([X,\epsilon][F,\epsilon]\Big)\]
which gives
\[c_2(X,Y;F) = c_0(X,Y) + (\delta\eta)(X,Y;F)=\frac{1}{8}{\rm Tr}\Big([[\epsilon,X],[\epsilon,Y]](\epsilon-F)\Big) \ .\]
An important consequence of this so called ``infinite charge
renormalization" is that the central extension 
$\widehat{L\mathfrak{g}}$ will be replaced by an abelian extension by the infinite-dimensional ideal ${\rm Map}(Gr_p, \mathbb{C})$,
\[0\to {\rm Map}(Gr_p, \mathbb{C})\to \widehat{L_q\mathfrak{g}}\to L_q\mathfrak{g} \to 0 \ .\]

\section{Lie algebra cocycles}

In this section we will construct Lie algebra cocycles for $\mathfrak{gl}_p$ for all $p\geq\frac 12$, 
which by restriction yield cocycles on $L_q\mathfrak{g}$. In particular, we show that the cocycles respect the 
decomposition of $L_q\mathfrak{g} = L_q\mathfrak{g}_+\oplus L_q\mathfrak{g}_-$ 
into positive and negative Fourier modes on the circle. The computation is done in the non-commutative BRST 
bicomplex, where the classical de Rham complex is replaced by a graded differential algebra 
$(\Omega, d)$. Let $\epsilon = \frac{D^q}{|D^q|}$ denote the sign operator on the circle satisfying 
$\epsilon=\epsilon^*$ and $\epsilon^2=1$. We introduce a family of vector spaces
\begin{eqnarray*} 
\Omega^{0}&=&\{X\in \mathcal{B}(\mathcal{H}) \ | \ X-\epsilon X \epsilon \in L_{2p}\}\\
\Omega^{2k-1}&=&\{X\in \mathcal{B}(\mathcal{H}) \ | \ X+\epsilon X \epsilon \in L_{\frac{2p}{2k}}, 
\ X-\epsilon X \epsilon \in L_{\frac{2p}{2k-1}}\}\\
\Omega^{2k}&=& \{X\in \mathcal{B}(\mathcal{H}) \ | \ X+\epsilon X \epsilon \in L_{\frac{2p}{2k}}, 
\ X-\epsilon X \epsilon \in L_{\frac{2p}{2k+1}}\}
\end{eqnarray*}
and set $\Omega = \oplus_{k=0}^\infty \Omega^{k}$. The exterior differentiation is defined by
\begin{equation*}dX = \begin{cases}  [\epsilon, X], & {\rm if} \ X \in \Omega^{2k} \\ 
 \{\epsilon, X\}, & {\rm if} \  X \in \Omega^{2k-1} \end{cases}\end{equation*}
and satisfies $d^2=0$. Here $\{X,Y\}=XY+YX$ is the anticommutator. The space $\Omega^k$  consists of linear 
combinations of $k$-forms
$X_0dX_1 dX_2 \dots dX_k$ with $X_i \in \Omega^0$.  
 If $X\in \Omega^{k}$ and $Y\in\Omega^{l}$, then 
\[XY\in\Omega^{k+l}, \ \ dX \in \Omega^{k+1}, \ \ d(XY) = (dX)Y+(-1)^kXdY\]
which follows by the generalized H\"older inequality for Schatten ideals. This ensures that $(\Omega, d)$ is 
a $\Bbb N$-graded differential algebra. 
Integration of forms is substituted by a graded trace functional
\begin{equation*} {\rm Str}(X) =  {\rm Tr}_C(\Gamma X), \ {\rm for \,\,   }  X\in \Omega^k  \ {\rm and   \,\, }  k\geq p ,
\end{equation*}
where $\Gamma$ is a grading operator on $\mathcal H$ and ${\rm Tr}_C(X) = \frac12 {\rm Tr}(X+\epsilon X\epsilon)$ 
is the conditional trace, \cite{Connes}. In case of an even Fredholm module ($k$ even), $\Gamma$ anticommutes with 
$\epsilon$ and in case of an odd Fredholm module ($k$ odd), $\Gamma=1$. 

Next we define the Lie algebra chain complex $(C, \delta)$. Since $\Omega^{0}= \mathfrak{gl}_p$ and $Gr_p \subset\Omega^{1}$, 
we interpret $B \in \Omega^{1}$ as generalized connection 1-forms and
define the infinitesimal gauge action by 
\[\Omega^{0}\times \Omega^{1}\to \Omega^{1}, \ \ \ (X,B) \mapsto \mathcal{L}_X B = [B,X] + dX \ .\]
Indeed for any $F\in Gr_p$ we have $F=g^{-1}\epsilon g$ for some $g\in GL_p$, since $GL_p$ acts transitively 
on the Grassmannian. Thus $F-\epsilon = g^{-1}dg$ corresponds to flat connections. The abelian group 
$ Map(\Omega^1,\Omega)$ is naturally a $\Omega^0$-module under the action
\[\Omega^0\times Map(\Omega^1,\Omega)\to Map(\Omega^1,\Omega), \ \ \ (X,f)\mapsto 
\mathcal{L}_Xf(B)=\frac{d}{dt}f\Big(e^{-tX}Be^{tX}+tdX\Big)\Big|_{t=0} \ .\]
Define the space of $k$-chains $C^k$ as alternating multilinear maps
\[\omega:\underbrace{\Omega^0\times\ldots\times\Omega^0}_{k}\to Map(\Omega^1,\Omega)\]
and set $C=\oplus_{k=0}^\infty C^k$. The coboundary operator is given by Palais' formula 
\begin{eqnarray*} 
\delta \omega(X_1,\ldots,X_k;B)&=&\sum_{j=1}^k(-1)^{j+1}
\mathcal{L}_{X_j}\omega(X_1,\ldots,\hat{X}_j,\ldots,X_k;B) \\
&&+\sum_{i<j}(-1)^{i+j}\omega([X_i,X_j], X_1,\ldots,\hat{X}_i,\ldots,\hat{X}_j,\ldots,X_k;B)
\end{eqnarray*}
and satisfies $\delta^2=0$. Here $\hat{X}_j$ means that the variable $X_j$ is omitted. 

This provides us with a double complex. For our purposes, we need cocycles of degree 2 in the 
Lie algebra cohomology. There is a straightforward way to compute such cocycles parametrized 
by \textit{flat} connections $B=F-\epsilon$. Given a bicomplex with commuting differentials, there is 
an associated singly graded complex with differential $D = \delta + (-1)^k d$. The Lie 
algebra 2-cocycle is given by
\begin{eqnarray*} 
\tilde c_p(X,Y; B) &=& \frac{2^{2p}}{(2p+1)} \ {\rm Str}\big(g^{-1}D g\big)^{2p+3}_{[2p+1,2]}(X,Y) \\
&=& 2^{2p} \ {\rm Str}\left(\sum_{k=0}^{p}(g^{-1}{\rm d} g)^{2p+1-k}(g^{-1} \delta g)(g^{-1}
{\rm d} g)^k ( g^{-1} \delta g)\right)(X,Y)\\
&=& 2^{2p} \ {\rm Str}\left( \sum_{k=0}^{p} (-1)^k \big(B^{2p+1-k}XB^kY-B^{2p+1-k}YB^kX\big)\right)
\end{eqnarray*}
for all $p\geq0$. By $(\ldots)_{[j,k]}$ we mean the component of degree $(j,k)$ in $(d, \delta)$ 
cohomology. That this is a cocycle follows from
\[\delta {\rm Str}(g^{-1}D g)^{2p+3}= {\rm Str}D (g^{-1}D g)^{2p+3} = 
{\rm Str}(g^{-1}D g)^{2p+4}=0 \]
where we have used that ${\rm Str}(g^{-1}D g)^{even}=0$. Although $\tilde c_p(X,Y; B)$ does not 
vanish when $X,Y\in L_q\mathfrak{g}_{-}$, we have:
\begin{theorem} \label{cocycles} The cocycle $\tilde c_p$ is cohomologous to a cocycle $c_p$ which has the property 
	that $ c_p(X,Y;B)=0$ if both $X$ and $Y$ are in $L_q\mathfrak{g_-}$ (or both in $L_q\mathfrak{g_+}$). 
\end{theorem}

\begin{remark} \rm Note that we can replace the abelian ideal of smooth functions of the variable $B$ 
by the space of functions on $L_qG$ by using the embedding of $L_qG$ to the Grassmannian
$Gr_p$ given by $g\mapsto B= g^{-1}[\epsilon, g].$ 
\end{remark}

\noindent The proof of Theorem \ref{cocycles} is by direct computation and is shifted to the Appendix.  
 
\begin{theorem} The cocycle $c_p$ in the Appendix, when restricted to the Lie
algebra of smooth loops, is cohomologous to the cocycle defining the standard central
extension of the loop algebra $L\mathfrak{g}$. \end{theorem}
\begin{proof} Define $\eta_p(X;B) = 2^{2p+1}\ \text{Tr}\big( \epsilon B^{2p+1} dX\big)$. Then
by direct computation, similar to the one in the Appendix and which will not
be repeated here,  
$$ c_{p+1}(X,Y;B)= c_p(X,Y;B)-(\delta \eta_p)(X,Y;B) \ .$$ 
Thus $ c_p$ is cohomologous to $ c_{p-1}$ and by induction to
the cocycle  $ c_0$. But 
$$ c_0(X,Y;B) = \frac12 \text{Tr}\big(XdY\big)$$
which is precisely the cocycle defining the central extension of the loop algebra. \end{proof}
We note that the 1-cochain $\eta(X;F)$ in Section 3 is simply the sum $ - \sum_{k=0}^{p-1} \eta_k(X;B)$, 
where $B=F-\epsilon$.

\begin{remark} \rm Since the Lie algebra cocycle vanishes on $L_q\mathfrak g_-$ one can define
generalized Verma modules as the quotient of the universal enveloping algebra $\mathcal U(
\widehat{L_q\mathfrak g})$ by the left ideal generated by $L_q\mathfrak g_-$ and by the elements
$h -\lambda(h)$ where the $h$'s are elements in a Cartan subalgebra $\mathfrak h$ of $\mathfrak g$ and 
$\lambda \in \mathfrak h^*$ is a weight. In the case of the central extension of the smooth loop
algebra, for dominant integral weights one can construct an invariant hermitean form in the
Verma module; there is a subquotient (including the highest weight vector) which carries
an irreducible unitarizable 
representation of the Lie algebra. However, in the case of $L_q\mathfrak g$ for $q < 1/2$ 
this is not possible: due to the large abelian ideal in the extension of $L_q\mathfrak g$
we cannot construct any invariant hermitean semidefinite form on the Verma module.
Whether this is possible at all is an open question, although we conjecture that the answer
is negative. 
\end{remark}

\section{Generalized supersymmetric WZW model}

Let us recall the construction of a family of supercharges $Q(A)$ for the supersymmetric 
Wess-Zumino-Witten model in the setting of representations of the
smooth loop algebra, \cite{Mic02}, \cite{Lan}, \cite{FHT03}.

Here we assume that $G$ is a connected, simply connected simple 
compact Lie group of dimension $N$. Let $\mathcal{H}_b$ denote the ``bosonic" Hilbert space carrying an 
irreducible unitary highest weight representation of the loop algebra $\widehat{L\mathfrak{g}}$ 
of level $k$. The level $k$ is a non-negative integer and we introduce $k' =  2\theta^2 k$, 
where $\theta$ is the length of the longest root of $G$. The generators of the loop algebra 
in Fourier basis $T^a_n$ satisfy
\[[T^a_n, T^b_m]= \lambda^{abc}T^c_{n+m}+\frac{k'}{4}n\delta^{ab}\delta_{n,-m}\] 
where $n\in\mathbb{Z}$ and $a=1,2,\ldots,N$. We fix an orthonormal basis $T^a$ in $\mathfrak{g}$,
with respect to the Killing form, so that the structure constants $\lambda^{abc}$ are completely 
antisymmetric and the Casimir invariant $C_2=\lambda^{abc}\lambda^{acb}$ equals $-N$. 
Moreover, we have 
\[(T^a_n)^*=-T^a_{-n} \ .\]
The fermionic Hilbert space $\mathcal{H}_f$ carries an irreducible representation of the canonical 
anticommutation relations (CAR)
\[\{\psi^a_n,\psi^b_m\}=2\delta^{ab}\delta_{n,-m}\]
where $(\psi^a_n)^*=\psi^a_{-n}$. The Fock vacuum is a subspace of $\mathcal{H}_f$ of dimension 
$2^{[\frac{N}{2}]}$. It carries an irreducible representation of the Clifford algebra spanned 
by the zero modes $\psi^a_0$ and lies in the kernel of all $\psi^a_n$ with $n<0$. The loop 
algebra $\widehat{L\mathfrak{g}}$ acts in $\mathcal{H}_f$ through the minimal representation of level 
$h^\vee$, the dual Coxeter number of $G$. 
The operators are explicitly realized as bilinears in the Clifford generators
\[K^a_n=-\frac{1}{4} \lambda^{abc}: \psi^b_{n-m}\psi^c_m : \]
and satisfy
\[[K^a_n, K^b_m]= \lambda^{abc}K^c_{n+m}+\frac{h'^{\vee}}{4}n\delta^{ab}\delta_{n,-m} \ ,\] 
where $h'^{\vee}= 2 \theta^2 h^{\vee}$. The normal ordering $: \ :$ indicates that operators with positive Fourier index are placed 
to the left of those with negative index. In case of fermions there is a change of sign, 
$:\psi^a_{-n}\psi^b_n:=-\psi^b_n\psi^a_{-n}$ if $n>0$. The full Hilbert space 
$\mathcal{H}=\mathcal{H}_b\otimes \mathcal{H}_f$ carries a tensor product representation of 
$\widehat{L\mathfrak{g}}$ of level $k+h^\vee$, with generators 
$S^a_n = T^a_n\otimes \mathbf{1}+\mathbf{1}\otimes K^a_n $. 

The supercharge operator is defined by
\[Q=i\psi^a_n\left(T^a_{-n}+\frac{1}{3}K^a_{-n}\right)\]
and squares to the free Hamilton operator $h = Q^2$ of the supersymmetric WZW model
\[h = -: T^a_nT^a_{-n}: +\frac{\bar{k}}{2}: n\psi^a_n\psi^a_{-n}: +\frac{N}{24} = 
h_b+2\bar{k}h_f+\frac{N}{24}\]
where $\bar{k}= \frac{k'+h'^{\vee}}{4}$. Interaction with external $\mathfrak{g}$-valued 1-forms 
$A$ on the circle is introduced by minimal coupling
\[Q(A) = Q+i\bar{k}\psi^a_{n}A^a_{-n}\]
where $A^a_n$ are the Fourier coefficients of $A$ in the basis $T^a_n$, satisfying 
$(A^a_n)^*=-A^a_{-n}$. This provides us with a family of self-adjoint Fredholm operators 
$Q(A)$ that is equivariant with respect to the action of $\widehat{LG}$,
\[S(g)^{-1}Q(A)S(g) =Q(A^g)\]
with $A^g= g^{-1}Ag+g^{-1}dg$. Infinitesimally this translates to
\[[S^a_n,Q(A)]=i\bar{k}(n\psi^a_n+\lambda^{abc}\psi^c_{n+m}A^b_{-m}) = -\mathcal{L}^a_n Q(A) \ .\]
where $\mathcal{L}^a_n$ denotes the Lie derivative (infinitesimal gauge transformation) in the 
direction $X=S^a_n.$ 
The interacting Hamiltonian $h(A)=Q(A)^2$ is given by
\[h(A)=h-\bar{k}\left(2S^a_nA^a_{-n}+\bar{k}A^a_{n}A^a_{-n}\right)=h +h_{int} \ .\]

Next we consider extending this construction to the fractional
setting.  As previously mentioned,
this necessarily entails certain regularization. Let us first denote
by  $S_0$ the representation of the smooth  loop algebra with
commutation relations  
\[[S_0(X),S_0(Y)]=S_0([X,Y])+c_0(X,Y) \ .\]
We proceed by adding a 1-cochain
\begin{eqnarray*} 
S(X) &=& S_0(X)+\eta(X;B) \\
Q &=&  Q_0 + \eta(\psi;B)
\end{eqnarray*}
where in component notation $\eta(\psi;B)= \psi^a_n
\eta(T^a_{-n};B) = \psi^a_n\eta^a_{-n}$ and we denote the original
supercharge by $Q_0$ from now on. 
Here $B=g^{-1}[\epsilon, g]$ parametrizes points on the Grassmannian, as
in Section 4, with $g\in L_q G.$ 
We write also $\mathcal L_X =
X^a_n \mathcal L^a_{-n}$ for an element $X\in L_q\mathfrak g.$ Adding
cochains of the type $\eta,$ which are functions of the variable $B$,
means that we are extending the original loop algebra by Fr\'echet
differentiable functions of $B$.  Since gauge transformations are
acting on $B$ by the formula $B\mapsto g^{-1}Bg + g^{-1}[\epsilon,g]$, or infinitesimally as 
$B\mapsto [B,X] + [\epsilon,X]$ for
$X\in L_q\mathfrak g,$ the commutator of $S(X)$ by any Fr\'echet
differentiable function $f$ of $B$ is given as
\[ [S(X), f(B)] = \mathcal L_X f(B) \ .\]
The other new commutation relations will be
\begin{eqnarray*} 
\text{[}S(X), S(Y) \text{]}&=& S(\text{[}X,Y\text{]})+c(X,Y;B) \\
\text{[}S(X),\psi(Y)\text{]} &=& \psi(\text{[}X,Y\text{]}) \\
\text{[}S(X),Q\text{]} &=& ic(X,\psi;B) \\
\{Q, \psi(Y)\} &=& 2i S(Y)
\end{eqnarray*}
where $c(X,Y;B)=c_0(X,Y)+(\delta\eta)(X,Y;B)$ converges for an
appropriate choice of $\eta,$ according to Theorem 4.3. 
Moreover, we set $h = Q^2$ where
\[ Q^2 = Q_0^2 - 2 S(\eta)-\eta^2 + i \mathcal L_{\psi}\eta(\psi;B)\ ,\]
$\eta^2= \eta^a_n\eta^a_{-n}$ and $S(\eta) = \eta^a_n S^a_{-n}.$ 
In Fourier basis, the generators $\big\{\psi^a_n, S^b_m, Q, h\big\}$ satisfy
the following commutation relations,
\begin{eqnarray*} 
\{\psi^a_n,\psi^b_m\} &=& 2\delta^{ab}\delta_{n,-m}\\
\text{[}S^a_n, S^b_m\text{]}&=& \lambda^{abc}S^c_{n+m}+c^{a,b}_{n,m}(B)\\
\text{[}S^a_n, \psi^b_m\text{]}&=& \lambda^{abc}\psi^c_{n+m}\\
\{\psi^a_n,Q\}&=& 2iS^a_n\\
\text{[}S^a_n,Q\text{]}&=& ic^{a,b}_{n,-m}(B)\psi^b_m\\
\text{[}\psi^a_n,h \text{]}&=& 2c^{a,b}_{n,-m}(B)\psi^b_m \\
\text{[}h,S^a_n\text{]}&=& 2S^b_m c^{a,b}_{n,-m}(B)-\psi^b_m\psi^c_p\mathcal{L}^c_{-p}c^{a,b}_{n,-m}(B)\\
\text{[}Q,h\text{]} &=& 0
\end{eqnarray*}
where $c^{a,b}_{n,m}(B)=c(S^a_n,S^b_m;B)$. For any Fr\'echet differentiable function $f=f(B)$,
\begin{eqnarray*} 
\text{[}S^a_n, f\text{]}&=& \mathcal{L}^a_{n} f\\
\text{[}Q,f\text{]}&=& i \psi^a_n\mathcal{L}^a_{-n} f\\
\text{[}h, f\text{]}&=& -2S^a_n\mathcal{L}^a_{-n} f+\psi^a_n\psi^d_q\mathcal{L}^d_{-q}\mathcal{L}^a_{-n} f \ .
\end{eqnarray*}
In the smooth case, $c^{a,b}_{n,m}(B)=kn\delta^{ab}\delta_{n,-m}$, one recovers the 
corresponding subalgebra of the superconformal current algebra, \cite{KT}. Let us consider highest weight 
representations of the loop algebra generated by $S$ and $S_0$ respectively. Since they differ by a coboundary, 
one can explicitly relate their vacua by restricting to the subalgebra of smooth loops 
$L\mathfrak{g}\subset L_q\mathfrak{g}$. Indeed, we have
\[S(X)|\Omega> = 0, \ \ \ S_0(X)|\Omega_0> = 0 \]
for all $X \in L\mathfrak{g}_{-}$, which implies
\[S(X)|\Omega_0>=\Big(S_0(X)+\eta(X;B)\Big)|\Omega_0>=\eta(X;B)|\Omega_0>\neq 0 \ .\]
However if $(\delta \eta)(X,Y;B)=0$ for all $X,Y\in L\mathfrak{g}_{-}$, then $\eta$ restricts 
to a 1-cocycle on $L\mathfrak{g}_{-}$ and can in fact be written $\eta(X;B) = \mathcal{L}_X \Phi(B)$ 
for some function $\Phi$ of the variable $B$ on the smooth
Grassmannian consisting of points $g^{-1}[\epsilon,g]$ for $g\in LG.$  
For the cochain $\eta$ in Theorem 4.3 one can choose $\Phi(B) \sim \text{Tr}\big ( \epsilon B^{2p+1} \big)$ 
and the vacua are linked according to
\[|\Omega> = e^{-\Phi(B)}|\Omega_0> \ .\]
Indeed for all $X \in L\mathfrak{g}_{-}$, we have 
\begin{eqnarray*}
S(X)|\Omega>  &= & e^{-\Phi(B)}\Big(S(X)-\mathcal{L}_X \Phi(B)\Big)|\Omega_0>\\
&=&  e^{-\Phi(B)}\Big(S_0(X)+\eta(X;B)-\mathcal{L}_X \Phi(B)\Big)|\Omega_0>=0 \ .
\end{eqnarray*}

\section{Twisted K-theory and the group $L_qG$}

We want to make sense of a family of supercharges $Q(A)$ which
transforms equivariantly under the action of the abelian extension
$\widehat{L_q G}$ of the fractional loop group $L_qG.$ This should
generalize the construction of the similar family in the case of
central extension of the smooth loop group. Let us recall the
relevance of the latter for twisted K-theory over $G$ of level
$k+h^\vee.$ We fix $G$ to be a simple compact Lie group throughout this
section.
One can think 
of elements in $K^*(G, k+h^\vee)$ as maps $f: \mathcal A \to Fred(\mathcal H),$ to 
Fredholm operators in a Hilbert space $\mathcal H,$ with the property $f(A^g)= 
\hat g^{-1} f(A) \hat g.$ Here $g\in LG$ and $\mathcal A$ is the space of
smooth $\mathfrak g$-valued vector potentials on the circle. The
moduli space $\mathcal A/\Omega G$ (where $\Omega G$ is the group of based
loops) can be identified as $G.$ Actually, one can
still use the equivariantness under constant loops so that we really
deal with the case of $G$-equivariant twisted K-theory $K^*_G(G,
k+h^\vee).$ For odd/even dimensional groups one gets elements in
$K^1/K^0.$ 
 
The real motivation here is to try to understand the corresponding
supercharge operator $Q$ arising from Yang-Mills theory in higher 
dimensions.  If $M$ is a compact spin manifold the gauge group $Map(M,G)$
can be embedded in $U_p$ for any $2p > \text{dim} M;$ this was used in \cite{MR} for 
constructing a geometric realization for the extension of $Map(M,G)$ arising from
quantization of chiral Dirac operators   in background gauge fields. This is an analogy
for our embedding of $L_qG$ in $U_p$ for $p > 1/2q,$ the index $1/2q$ playing the
role of the dimension of $M.$

The more modest aim here is to show that there is a true family of Fredholm operators which 
transforms covariantly under $L_qG$. The operators are parametrized by 1-forms on
$L_qG$ and generalize the family of Fredholm operators $Q(A)$ from the
smooth setting to the fractional case. Hopefully, this will help us understand the renormalizations
needed for the corresponding problem in gauge theory on a manifold $M.$ 

Let us denote by $L^cG$ the Banach-Lie group of continuous loops in $G.$
The natural topology of $L^cG$ is the metric topology defined as 
$$d(f,g)= \sup_{x\in S^1}\, d_G(f(x),g(x))$$
where $d_G$ is the distance function on $G$ determined by the Riemann metric. Local charts
on $L^cG$ are given by the inverse of the exponential function; at any point $f_0\in L^cG$ we can map
a sufficiently small open ball around $f_0$ to an open ball at zero in $L^c\frak{g}$ by 
$f\mapsto \log(f_0^{-1}f).$ 

In the smooth version $LG$ the topology is locally  given by the topology on $L\frak g;$
the topology of the vector space $L\frak g$ is defined by the family of seminorms 
$||X||_n= \sup_{x\in S^1} |X^{(n)}(x)|$ for a fixed norm $| \cdot|$ on $\frak g.$ 
More precisely, we can define a family of distance functions on $LG$ by
$$d_0(f,g) = \sup_{x\in S^1}\,d_G(f(x),g(x)) $$ 
for $n=0$, and 
$$d_n(f,g)=  \sup_{x\in S^1} |f^{(n)}(x) -g^{(n)}(x)| $$
for $n>0$, where $f^{(n)}$ is the n:th derivative with respect to the loop parameter: We identify
the first derivative as a function with values in $\frak{g}$ by left translation $f'\mapsto f^{-1}f'$ and then
all the higher derivatives are $\frak{g}$-valued functions on the circle.

The metric is then defined as
$$d(f,g) = \sum_{n\geq 0} \frac{d_n(f,g)}{1+d_n(f,g)} 2^{-n}\ .$$

The subgroup $LG \subset L^cG$ is dense in the topology of the latter.
For this reason the cohomology of $L^cG$ is completely determined by restriction to
$LG.$
Actually, we have a stronger statement:

\begin{lemma} (Carey-Crowley-Murray) The group $L^cG$ of continuous loops is homotopy equivalent to the smooth loop group $LG.$  \end{lemma}
\begin{proof} We may assume that $G$ is connected; otherwise, one repeats the proof for each component of $G.$ When $G$ is connected the full loop group is a product of $G$ and the 
group $\Omega G$ of based loops (whether continuous, smooth, or of type $L_q$). So we restrict
to the group of based loops. 
As shown in \cite{CCM}, the groups $\Omega^cG$ and $\Omega G$ are weakly homotopic, i.e. the
inclusion $\Omega G\subset \Omega^cG$ induces an isomorphism of the homotopy groups. According to
the Theorem 15 by R. Palais, \cite{Pa}  a weak homotopy equivalence of metrizable manifolds implies homotopy equivalence. Actually, in \cite{CCM} the authors use the CW property of
the loop groups for the last step. The CW property in the case of $\Omega^cG$ is a direct consequence of Theorem 3 in \cite{Mi}. \end{proof}

\begin{lemma} The group $L_qG$ is homotopy equivalent to $LG$ and thus also to $L^c G.$
\end{lemma}
\begin{proof}  The proof in \cite{CCM} can be directly adapted from the smooth setting 
to the larger group $L_qG.$ Let $M$ be a compact manifold with base point $m_0$ and 
$C((I^n, \partial I^n), (M,m_0))$ the set of continuous maps from the n-dimensional unit
cube to $M$ such that the boundary of the cube is mapped to $m_0.$ The key step in their proof is the observation that in the homotopy
class of any map $g\in C((I^n, \partial I^n), (M,m_0))$ there exists a smooth map; in addition, a 
homotopy can be given in terms of a differentiable map. Taking $M=G$ 
and thinking of $g$ as a representative for an element in the homotopy group $\pi_{n-1}(\Omega^c G).$
Since $g$ is homotopic to a smooth map, it also represents an element in the smooth homotopy
group of $\Omega G$ and thus also an element in the $(n-1)$:th homotopy group of
$\Omega_q.$ In addition, a continuous homotopy is equivalent to a smooth homotopy.
The embedding of $LG \subset L_qG$ is continuous in their respective topologies [this follows
from the Sobolev norm estimates in the proof of Proposition 3.1] and therefore
the representatives for the homotopy groups of the former are mapped to representatives of
the homotopy groups of the latter. \end{proof}

Let $F: L^cG \to LG$ be a smooth homotopy equivalence. We define 
$$ \theta(f;g) = F(f)^{-1} F(fg)$$
for $f,g\in L^cG.$ This is a 1-cocycle in the sense of
$$ \theta(f;gg') = \theta(f;g) \theta(fg;g')\ .$$ 
For any $g\in LG$ we then have
$$\hat{\theta}(f;g)^{-1} Q_0\hat{\theta}(f;g) = Q_0 + i\bar{k}<\psi, \theta(f;g)^{-1} \partial \theta(f;g)>$$
where $\hat \theta $ is the lift of $\theta$ to the central extension $\widehat{LG}.$ 
Here $\partial$ is the differentiation with respect to the loop parameter.

Since the homotopy $F$ is smooth we can define a Lie algebra cocycle
$$d\theta(f;X) = \frac{d}{dt}|_{t=0} \theta(f; e^{tX})$$
with values in $L\frak g$ for $X\in L^c\frak g.$ This is a 1-cocycle in the sense that
$$d\theta(f;[X,Y]) +\mathcal L_X d\theta(f;Y)-\mathcal L_Y d\theta(f;X) = [d\theta(f;X), d\theta(f;Y)]\ .$$

Let next $Q(A) = Q_0 +i\bar{k}  <\psi, A>$ be a perturbation of $Q_0$ by a  function 
$A:L^cG \to L^c\frak g.$  The group $L^cG$ acts on $A$ by right translation,
$(g\cdot A)(f) = A(fg).$ Denote by $\hat \Theta(g)$ the operator consisting of the right translation on
functions of $f$ and of $\hat{\theta}(\cdot;g)$ acting on values of functions in the Hilbert space
$\mathcal H.$ Then 
$$ \hat\Theta(g)^{-1} Q(A) \hat\Theta(g) = Q(A^g)\ ,$$
where
\begin{eqnarray}\label{gaugetransf}
 (A^g)(f) = \theta(f;g)^{-1}A(fg)\theta(f;g) + \theta(f;g)^{-1} \partial\theta(f;g)\ . \end{eqnarray}
Since the group $LG$ acts in $\mathcal H$ through its central extension $\widehat{LG},$ the Lie
algebra $L^c\frak{g}$ acts through its abelian extension by $Map(L^cG, i\mathbb R),$ the extension
being defined by the 2-cocycle
$$ \omega(f;X,Y) = [\widehat{d\theta}(f;X), \widehat{d\theta}(f;Y)] - \widehat{d\theta}(f;[X,Y])
-\mathcal L_X \widehat{d\theta}(f;Y) +\mathcal L_Y\widehat{d\theta}(f;X)\ .$$

For an infinitesimal gauge transformations $X$, the formula (\ref{gaugetransf}) leads to
\begin{eqnarray}\label{gauge} \delta_X A = [A, d\theta(f;X)] + \partial \theta(f;X) + \mathcal L_X A \end{eqnarray}
which should be compared with $\delta_X A = [A,X] + \partial X$ in the smooth case, for
constant functions $A: LG  \to L\frak g.$ 

The Lie algebra cohomology of any Lie algebra, with coefficients in the module of smooth
functions on the Lie group $G$, is by definition the same as the de Rham cohomology of
$G$. Indeed, take a cocycle $c$ in de Rham cohomology on $G$. It is an alternating multilinear form on
on the space of vector fields on $G$, with values in the space of smooth functions on $G$.
We can restrict it to left invariant vector fields on $G$; but the left invariant vector fields
are just elements of the Lie algebra of $G$. So we obtain an alternating multilinear form on
the Lie algebra of $G$, with values in the space of smooth functions on $G$. Looking at the definition
of a de Rham cocycle (in terms of smooth vector fields) one sees that the cocycle condition is exactly the
same as the Lie algebra cocycle condition, with values in $C^{\infty}(G)$, with the standard action of
the Lie algebra (as derivations) on functions. Thus we get a linear map from the space of de Rham cocycles 
to Lie algebra cocycles. Exact cocycles on the de Rham side map to exact Lie algebra cocycles, so we have 
a map between the cohomologies. This is an isomorphism since the de Rham forms are uniquely determined by the
restriction to left invariant vector fields (at each point on G the left invariant vector fields form a
basis).

We apply this to the loop group $LG$ together with the fact that the standard central extension
of $L\frak g,$ when viewed as a left invariant 2-form on $LG,$ generates $H^2(LG, \mathbb R).$ Thus
the Lie algebra cohomology of $L\frak g$ with coefficients in the module of smooth functions
on $LG$ is one dimensional.  

By  the Lemma 6.2 we have 

\begin{lemma} For a simple compact Lie group $G$ the Lie algebra cohomology in degree 2 of $L^c\frak{g}$ with coefficients in the module of smooth
functions on $L^cG$ (with respect to the Banach manifold structure) is one dimensional and the class
of a cocycle is fixed by the restriction to the smooth version $LG.$  \end{lemma}

\begin{remark} \rm We can easily produce an explicit homotopy connecting the standard 2-cocycle on $L\mathfrak g$
to the restriction of the 2-cocycle on $L^c \mathfrak g$ to the smooth subalgebra $L\mathfrak g.$
Let $F_s : LG \to LG$ be any one-parameter family of maps connecting the identity $F_1$ to the
map  $F_0= F\circ i:  LG\to LG$ where $i:LG \to L^c G$ is the inclusion and $F:L^cG \to LG$ is  the homotopy equivalence used before,
$0\leq s\leq 1.$   On $LG$ we define
$$\eta =
 \int_0^1 \text{Tr} \Big( [\epsilon, F_s^{-1}\delta F_s]   F_s^{-1} \partial_s F_s\Big) \ .$$
Denote $c_0(X,Y) = \frac12 \text{Tr}\big( X[\epsilon, Y]\big)$, the standard Lie algebra 2-cocycle on $L\mathfrak g,$
and $c_1(X,Y) = \frac12 \text{Tr}\big( d\theta(f;X) [\epsilon, d\theta(f;Y)]\big)$ the 2-cocycle coming from the
homotopy $F_0.$ Then one checks that, for the restriction of $c_1$ to $LG,$
$$c_0 -c_1 = \delta\eta \ .$$
\end{remark}

\begin{proposition} The unitary subgroup  $U_p\subset GL_p$ for each $p\geq \frac 12$ is smoothly homotopic to
$U_{\frac 12}.$  \end{proposition}

\begin{proof} We prove the claim inductively by constructing a homotopy equivalence from
$U_p$ to $U_{p/2}$ for all $p\geq 1$. Let $g\in U_p$ with
$$g= \left( \begin{matrix} \alpha & \beta \\ \gamma & \delta \end{matrix}\right).$$
Let us denote $x= \alpha\gamma^* -\beta\delta^* \in L_{2p}.$  Define
the unitary operator $h(g)$
$$h(g) = \exp\left( \begin{matrix}  0 & -x/2 \\ x^*/2 & 0 \end{matrix}\right).$$
Let $F(g) = h(g)^{-1} g.$ Using $x^2 \in L_p,$ by a direct computation one checks that the upper right block in this
operator is equal to
$$ \beta +\frac12(\alpha\gamma^* -\beta\delta^*) \delta \text{ mod } L_p\ .$$
Using the unitarity relations $\alpha^*\beta +\gamma^*\delta =0$ and $\beta^*\beta +\delta^*\delta
=1$ we see that the above operator is equal to $\beta\beta^*\beta.$ Since $\beta\in L_{2p},$ by
the operator H\"older inequalities $\beta\beta^*\beta$ is in $L_{2p/3} \subset L_{p}.$ 
Thus $F(g) \in U_{p/2}.$ This map is homotopic to the identity map in $U_p:$ one just needs 
to replace the operator $x$ by $tx$ where $0\leq t \leq 1.$ At $t=0$ we then have $F_t(g) =g.$ 
Since the blocks of $F(g)$ are rational functions in the blocks of $g$ without singularities,
the map $g\mapsto F(g)$ is smooth in the natural Banach manifold structures of the groups
$U_p$ and $U_{p/2}.$  The same argument shows that the identity map on $U_r$ contracts 
to  a map to the subgroup $U_{p/2}$ for any $p\geq r \geq p/2.$ 
\end{proof}

According to \cite{PrSe} the restriction of the standard 2-cocycle of the Lie algebra of $U_1$ 
(denoted by $U_{res}$ in \cite{PrSe}) to the subalgebra $L\frak{g}$ is the cocycle defining an affine 
Kac-Moody algebra. In combination with  Lemma 6.3 and Proposition 6.5 we conclude:

\begin{proposition} The restriction from $L^cG$ to $L_qG$ of the Lie algebra cocycle $\omega$ is equivalent to the Lie algebra cocycle  of Theorem 4.1  (obtained by restriction of $c_p$ from $GL_p$ to $L_q G$ with $p = \text{max}\left\{\frac 12, \frac{1}{2q}\right\}$).
\end{proposition}

\begin{remark} \rm The discussion in Remark 6.4 applies here as well. By a similar formula one can 
produce an explicit homotopy between the standard 2-cocycle on the Lie algebra of $U_{\frac 12}$ 
and the cocycle obtained from the restriction from $U_p$ to the subgroup $U_{\frac 12}$, for  
$p\geq \frac 12.$ 
\end{remark}

We close this section by showing that in the fractional setting, the transformation \eqref{gauge} of the field $A:L_qG\to L_q\mathfrak{g}$ under $L_qG$ has a more geometric interpretation.

Using the inner product $<X,Y>= \int_{S^1} \left(X(x), Y(x)\right)_{\mathfrak{g}}dx$ in $L_q\frak{g},$ where $(\cdot,\cdot)_\mathfrak{g}$ is an invariant inner product in $\frak{g},$  and the fact a Lie group is a parallelizable manifold we can think of $A$ as a
1-form on the loop group $L_qG.$  We denote by $\mathcal A$ the space of all Fr\'echet differentiable 1-forms on
$L_qG.$ There is a  circle bundle $P$ over $L_qG$
with connection and curvature; the  
curvature $\omega$ is given by the cocycle of the abelian extension, $\omega(g;X,Y)$
where the elements $X,Y\in T_g (L_qG)$  are
identified as left invariant vector fields (elements of the Lie
algebra of $L_qG$). Let us denote by $\widehat{L_q G}$ the
extension of $L_qG$ by the abelian normal subgroup
$Map(L_qG,S^1)$ corresponding to the given Lie algebra extension. 

Conversely, starting from the abelian extension $\widehat{L_qG},$ viewed as a principal
bundle over $L_qG$ with fiber $Map(L_qG, S^1),$ we can recover the geometry of the 
the circle bundle $P.$  The connection in $P$ is obtained as follows. First, the connection form
in $\widehat{L_qG}$ is given as 
$$\Psi = Ad_{\hat{g}}^{-1} \text{pr}_c (d\hat{g}\hat{g}^{-1}) $$
where $\text{pr}_c$ is the projection onto the abelian ideal $Map(L_qG, i\mathbb{R}).$ 

\begin{remark} \rm In the case of the central extension $\widehat{LG}$ of $LG$ the adjoint action on
$\text{pr}_c (\bullet)$ is trivial but in the case of the abelian extension of $L_qG$ it is needed
in order to guarantee that the connection form is tautological in the vertical directions in the
tangent bundle $TP.$
\end{remark}

The connection $\nabla$ on $P$ is then defined using the identification of
$P$ as the subbundle of $\widehat{L_qG} \to L_qG$ with fiber 
$S^1$ consisting of constant functions in $Map(L_qG, S^1),$ that is, at each base point $g\in L_qG$
we have the homomorphism $\gamma$ sending a function $f\in Map(L_qG,S^1)$ to the value
of $f$ at the neutral element.
 This induces a homomorphism $d\gamma: Map(L_qG, i\mathbb{R}) \to i\mathbb{R}$
defining the $i\mathbb{R}$-valued 1-form $\psi=d\gamma \circ \Psi$, defining the covariant differentiation
$\nabla$ in the associated complex line bundle $L.$

An arbitrary connection in the bundle $P$ is then written as a sum
$\nabla + A$ with $A\in \mathcal A.$

\begin{theorem} The transformation of $\nabla+A$ under an infinitesimal right action of
 $(X, \alpha)\in \widehat{L_q \frak{g}}$ induces the action $\delta A = \omega_g(X, \bullet) + {\mathcal L}_X A
 +{\mathcal L}_{\bullet} \alpha.$ Here the Lie derivative acting on the from $A$ is composed from the directional derivative (by the left invariant vector field $X$)
 on the argument of $A$ and the commutator $[A,\theta(g;X)].$
 That is, under the infinitesimal transformation $\delta_{(X,0)}$ the transformation is the same as in (\ref{gauge}) using the identification of elements in $L_q\frak{g}$
 as elements in its dual through the inner product $<\cdot,\cdot>.$ 
  The function $\alpha$ can be interpreted as an infinitesimal
 gauge transformation in the line bundle $L$ over $L_qG.$ 
 
\end{theorem}

\begin{proof} Let $(X,\alpha) \in \widehat{L_q\frak g}.$ Infinitesimally, the right action is generated by left invariant vector fields. Thus the 
infinitesimal shift in the direction $(X,\alpha)$,  of the connection evaluated at the tangent vector $Y$ 
is the commutator $[\nabla_Y,\nabla_X +\alpha].$ The first term in the commutator  is equal to the curvature of the circle bundle
evaluated in the directions $X,Y.$ This is in turn equal to 
$\omega_g(X,Y).$  Since the covariant  derivative $\nabla_Y$ acts as the Lie derivative $\mathcal L_Y$ 
on functions, we obtain 

\begin{equation} \label{covariance}
\delta_{(X,\alpha)}  (\nabla +A) = \omega_g(X, \bullet) + {\mathcal L}_X A + {\mathcal L}_{\bullet} \alpha \ ,
\end{equation}
\end{proof}
It follows that we can morally interpret $Q(A)$ as a Dirac operator on the loop
group $L_qG$ twisted by a complex line bundle.

In the case of a central extension the above formulas (on level $k+h^\vee$)
reproduce the classical gauge action, after identification of the
dual  $L\bold{g}^*$ as $L\bold{g}$ using the Killing form,
\begin{equation}
 A \mapsto [A,X] + \frac14(k'+h'^{\vee}) dX,\end{equation}
for a left invariant 1-form $A$.  Namely, the right Lie derivative in  (\ref{covariance}), when acting
on left invariant forms, produces the commutator term $[A,X]$ and whereas the shift by $dX$ is
coming from the central extension $\omega_g(X, \bullet) = c_0(X,\bullet).$ The third term on the 
right-hand-side of the equation (\ref{covariance}) is absent since in case of central extension we can take
$\alpha=$ constant.

Going the other way, starting from the action on $\mathcal A$ in terms of the
cocycle $\omega,$ we may take the restriction to the Lie algebra of
smooth loops and the connection $\nabla$ becomes a sum
$$\nabla = \nabla_0 + \eta$$ 
where $\nabla_0$ is the connection in $P$ defined by the standard
central extension of the loop group and $\eta$ is a 1-form on $LG.$ The
form $\eta$ is actually the 1-cochain relating the central extension to the
abelian extension of $LG.$
Writing now $a = \eta + A$ and $S(X) = S_0(X) + \eta(X)$ we recover
the gauge action formula $a\mapsto [a,X] + \frac14(k'+ h'^{\vee}) dX$ with respect to
$S_0(X).$ Thus the addition of $\eta$ can be viewed as a
renormalization needed for extending from the case of smooth loop algebra to
the fractional loops.

\section*{Acknowledgements}
The authors would like to thank the Erwin Schr\"odinger Institute
for Mathematical Physics for hospitality were this work was initiated.
The work of the second author was partially supported by the grant
516 8/07-08 from Academy of Finland.

\vskip 0.3in
\noindent \textsc{Appendix:} Proof of Theorem \ref{cocycles}

\vskip 0.2in

\noindent Define a 1-cochain by 
\[\tilde \eta_p(X;B) = {\rm Str}\Big(B^{2p+1}X\Big)\]
for $p\geq 0$. Using Palais' formula the coboundary is given by
\begin{eqnarray*} 
(\delta \tilde \eta_p)(X,Y;B)&=& \sum_{k=0}^{2p}{\rm Str}\Big( B^k[B,X]B^{2p-k}Y-B^k[B,Y]B^{2p-k}X\Big)  \\
&& + \sum_{k=0}^{2p}{\rm Str}\Big(B^k {\rm d}XB^{2p-k}Y-B^k {\rm d}YB^{2p-k}X\Big)- {\rm Str}\Big(B^{2p+1}[X,Y]\Big) \ .
\end{eqnarray*}
By cyclicity of trace we can rewrite the first sum, 
\begin{equation}\label{eq1}
	\sum_{k=0}^{2p}{\rm Str}\Big( B^k[B,X]B^{2p-k}Y-B^k[B,Y]B^{2p-k}X\Big) = \end{equation}
\[={\rm Str}\Big( B[X,B^{2p}Y]-B[Y,B^{2p}X]\Big) +
\sum_{k=1}^{2p}{\rm Str}\Big( B^k[B,X]B^{2p-k}Y-B^k[B,Y]B^{2p-k}X\Big) \ .\]
This can be simplified further. Using $[a,bc] = [a,b]c + b[a,c]$ repeatedly we get  
\begin{equation}\label{eq2}
	B[X,B^{2p}Y]= B^{2p+1}[X,Y]+\sum_{k=0}^{2p-1}B^{2p-k}[X,B]B^kY \end{equation}
and inserting \eqref{eq2} into \eqref{eq1} yields
\[2 {\rm Str}\Big( B^{2p+1}[X,Y]\Big) +\sum_{k=0}^{2p-1} 
{\rm Str}\Big(B^{2p-k}[X,B]B^kY- B^{2p-k}[Y,B]B^kX\Big) \]
\[+ \sum_{k=1}^{2p}{\rm Str}\Big( B^k[B,X]B^{2p-k}Y-B^k[B,Y]B^{2p-k}X\Big) = 
2 {\rm Str}\Big( B^{2p+1}[X,Y]\Big) \ .\]
Thus the expression for the coboundary $\delta \tilde \eta_p$ reduces to
\begin{equation}\label{eq3}(\delta \tilde \eta_p)(X,Y;B)= {\rm Str}\left( B^{2p+1}[X,Y] + 
\sum_{k=0}^{2p} B^k {\rm d}XB^{2p-k}Y-B^k {\rm d}YB^{2p-k}X\right) \ . \end{equation}
We split the remaining sum into even and odd powers. When $k=2m$ is even,
\[\sum_{m=0}^{p} {\rm Str}\Big(B^{2m} {\rm d}XB^{2p-2m}Y-B^{2m}{\rm d}YB^{2p-2m}X\Big)=\]
\begin{equation}\label{eq4}= \sum_{m=0}^{p} {\rm Str}\Big(B^{2m} {\rm d}XB^{2p-2m}Y+B^{2m} YB^{2p-2m}{\rm d}X\Big)\end{equation}
and when $k=2m-1$ is odd,
\[\sum_{m=1}^{p} {\rm Str}\Big(B^{2m-1} {\rm d}XB^{2p-2m+1}Y-B^{2m-1}{\rm d}YB^{2p-2m+1}X\Big)=\]
\begin{equation}\label{eq5}= \sum_{m=1}^{p} {\rm Str}\Big(B^{2m-1}XB^{2p-2m+2}Y-B^{2m}XB^{2p-2m+1}Y\Big)  \ ,\end{equation}
where we have used ${\rm d}B^{2m} = 0$, ${\rm d}B^{2m+1} = -B^{2m+2}$ and ``integration by parts" using
\begin{eqnarray*}
{\rm d}(B^{2m-1}XB^{2p-2m+1}Y)&=& -B^{2m}XB^{2p-2m+1}Y-B^{2m-1}{\rm d}XB^{2p-2m+1}Y+\\
&& + B^{2m-1}XB^{2p-2m+2}Y + B^{2m-1}XB^{2p-2m+1}{\rm d}Y \ .
\end{eqnarray*}
Finally we show that the sum in $\tilde c_p(X,Y; B)$ for values $k\geq1$ equals \eqref{eq5}, up to the normalization factor $2^{2p}$;
\[\sum_{k=1}^{p} (-1)^k {\rm Str}\Big( B^{2p-k+1}XB^kY-B^{2p-k+1}YB^kX\Big) =\]
\[= \sum_{m=1}^{p/2}{\rm Str}\Big( B^{2p-2m+2}YB^{2m-1}X-B^{2p-2m+2}XB^{2m-1}Y\Big) \]
\[+ \sum_{m=1}^{p/2}{\rm Str}\Big( B^{2p-2m+1}XB^{2m}Y-B^{2p-2m+1}YB^{2m}X\Big) \]
\begin{equation}\label{eq6} =\sum_{m=1}^{p/2} {\rm Str}\Big( B^{2m-1}XB^{2p-2m+2}Y-B^{2m}XB^{2p-2m+1}Y\Big)  \end{equation}
\[+ \sum_{m=1}^{p/2}{\rm Str}\Big( B^{2p-2m+1}XB^{2m}Y-B^{2p-2m+2}XB^{2m-1}Y\Big) \ .\]
Shifting the index $m=n-\frac{p}{2}$ in the last sum in \eqref{eq6} we get
\[\sum_{n=(p+2)/2}^{p} {\rm Str}\Big( B^{3p-2n+1}XB^{2n-p}Y-B^{3p-2n+2}XB^{2n-p-1}Y\Big)\]
and reversing the summation order using
\[\sum_{\alpha}^{\beta}x^{\pm 2n}= \sum_{\alpha}^{\beta} x^{\pm 2(\alpha+\beta)\mp2n}\]
this becomes
\[\sum_{n=(p+2)/2}^{p}{\rm Str}\Big(  B^{2n-1}XB^{2p-2n+2}Y-B^{2n}XB^{2p-2n+1}Y\Big) \ .\]
Altogether this is precisely the odd part of the sum in \eqref{eq3}. Now we define a new cocycle
\begin{eqnarray*}
c_p(X,Y;B) &=& \tilde c_p(X,Y;B)- 2^{2p} (\delta \tilde \eta_p)(X,Y;B) \\
&=& -2^{2p}\sum_{m=0}^{p} {\rm Str}\Big(B^{2m} {\rm d}XB^{2p-2m}Y-B^{2m}{\rm d}YB^{2p-2m}X\Big)\\
&=& - 2^{2p}\sum_{m=0}^{p} {\rm Str}\Big(B^{2m} {\rm d}XB^{2p-2m}Y+B^{2m} YB^{2p-2m}{\rm d}X\Big) \ .
\end{eqnarray*}
We note that this cocycle vanishes when both $X,Y$ have only negative (or positive) Fourier components, because the trace of $B^n{\rm d}XB^{2m}{\rm d}Y$ and $B^{2m}{\rm d}XB^{2n}Y$ are zero when $X,Y$ are upper (or lower) triangular matrices with respect to the $\epsilon$ grading.  

\bibliographystyle{amsplain}

\end{document}